\documentclass{amsart}
\usepackage[utf8]{inputenc}
\usepackage{amsmath}
\usepackage{graphicx}
\usepackage{amssymb}
\usepackage{caption} 
\usepackage{hyperref}
\usepackage{tikz-cd}

\newtheorem{theorem}{Theorem}[section]
\newtheorem{corollary}{Corollary}[theorem]

\newtheorem{definition}{Definition}[section]
\newtheorem{conjecture}{Conjecture}[section]
\newtheorem{remark}{Remark}[section]

\title{Lorenz equations and the figure eight knot}
\author{Yara Hatoom}
\date{\today}

\begin{document}

\begin{abstract}
Lorenz equations were first presented in 1963 by Edward Lorenz, they depend on three real positive parameters. For some of these parameters which are called T-points, there are two heteroclinic orbits connecting the three singular points in the equations. The heteroclinic connections can be extended into an invariant curve passing through infinity. 
We consider the system at the second T-point parameter, and develop a geometric model for the flow that simulates the Lorenz dynamics there.  We show that the model contains infinitely many periodic orbits, and that as knots they are all positive, prime and fibered.
\end{abstract}

\maketitle

\section{Introduction}

The Lorenz equations are a family of three-dimensional ordinary differential equations which model unpredictable behaviour mimicking phenomena found in weather. The equations are given by:

\begin{equation}
\begin{gathered}
\begin{aligned}
    &\frac{dx}{dt}=\sigma(y-x)\\
    &\frac{dy}{dt}=rx-y-xz\\
    &\frac{dz}{dt}=xy-\beta z\\   
\end{aligned}
\end{gathered}
\end{equation}

The parameters $\sigma$, $r$ and $\beta$ are physical parameters, all three of them are taken to be positive. The equations were first studied by Lorenz \cite{lorenz1963deterministic} especially at the parameter values $\sigma=10, \beta=\frac{8}{3}, r=28$ which are called the classical parameters. It has been numerically proven in \cite{tucker1999lorenz} that at the classical parameter values, the flow converges almost always onto the famous butterfly strange attractor.\\

The Lorenz system has a symmetry of rotation about the $z$-axis $(x,y,z) \rightarrow (-x,-y,z)$, and for $r>1$ the system has three singular points, the origin and the points $(\pm{\sqrt{\beta(r-1)}},\pm{\sqrt{\beta(r-1)}},r-1)$ which we denote by $p^+, p^-$. These points $p^\pm$ are called the wing centers, as at the classical parameter values, where we have a butterfly shaped attractor, the points $p^\pm$ are at the centers of the butterfly attractor wings.
For many parameter values, all three singularities are of saddle type, the origin has a one dimensional unstable manifold on the other hand $p^\pm$ have one dimensional stable manifold. For some parameters which are called T-points, there are two heteroclinic orbits which join the three fixed points, one flows from the origin to $p^+$ and the other flows from the origin to $p^-$.\\

By \cite{lorenz1963deterministic} (see also \cite{sparrow2012lorenz}), there exists an ellipsoid around the origin which is transverse to the flow so that orbits only enter the region bounded by the ellipsoid and never leave it. This allows us to compactify $\mathbb{R}^3$, adding $\infty$ to be a singular point that is a source, obtaining a $C^1$ flow on $S^3$.\\

At a T-point, we consider the closed curve which passes through the three singular points and $\infty$. The curve is defined as follows; take the heteroclinic connections which join the singular points $p^\pm$ with the origin, we can extend these symmetric connections to an invariant closed curve which passes through $\infty$ as has been shown in \cite{Pinsky2023}.\\

For each T-point this curve is a different type of knot. At the first T-point the curve is the trefoil knot. At the second T-point the curve is the figure-eight knot. The case with the trefoil knot has been studied, there is a geometric model and we know how to classify the periodic orbits of the flow. In this research we focus on the figure-eight knot case \ref{figI}.\\

\begin{figure}[!h]
    \centering
    \includegraphics[width=7cm]{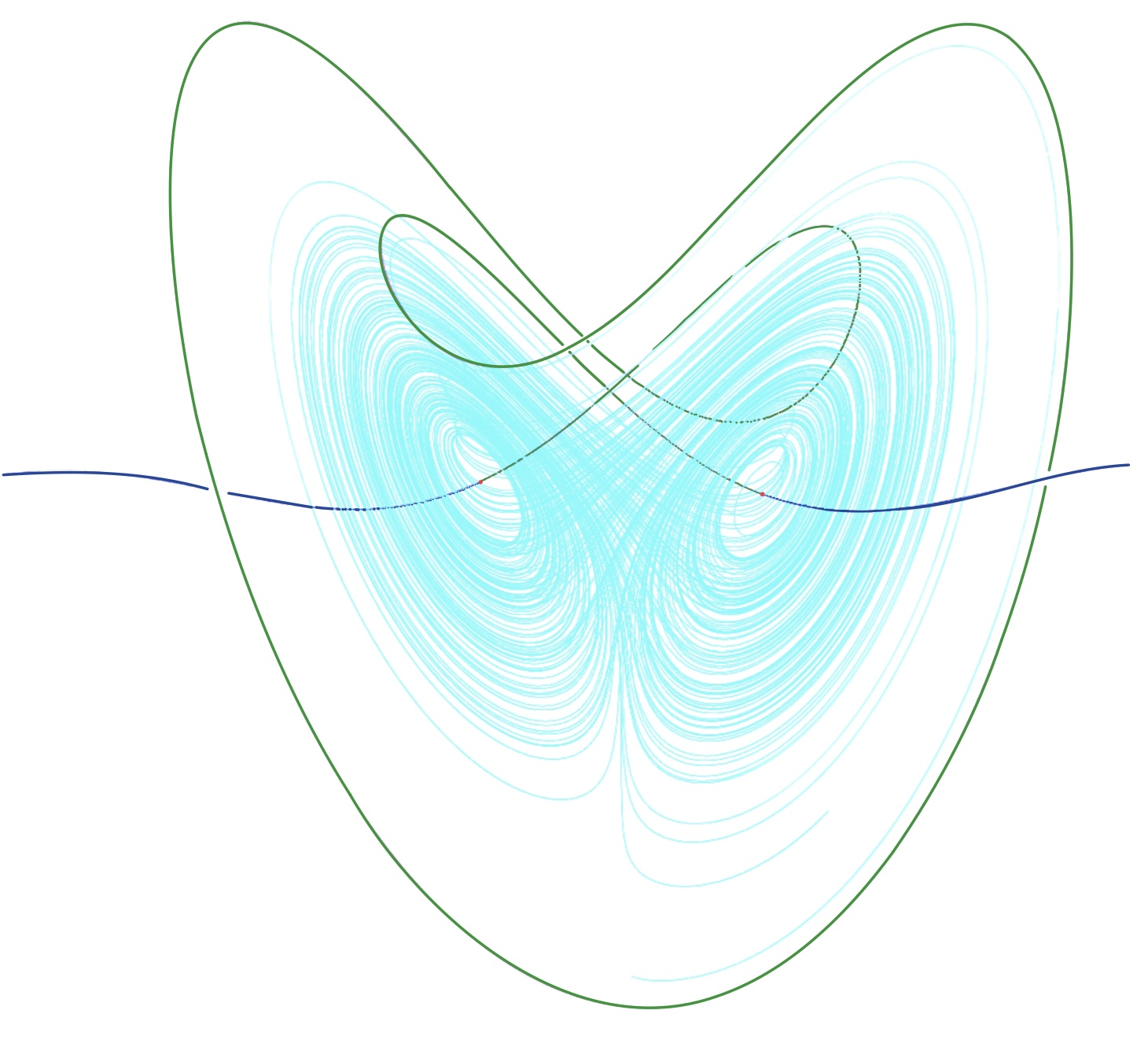}
    \caption{By numerical studies the second T-point is located at $(r,\sigma)\approx (85.0292,11.8279)$ for $\beta=\frac{8}{3}$}
    \label{figI}
\end{figure}

In this work we develop a geometric model for the second T-point parameter, by describing the first return map in the second T-point, obtaining a flow  defined on $S^3$ that is piecewise linear, has a global cross section, has four singularities with the same local behaviour as in Lorenz,  and contains a figure eight knot consisting of invariant heteroclinic orbits connecting all four singular points.

The periodic orbits for this geometric model are all contained in a two dimensional template, similar to the Lorenz template albeit more complicated. We prove the following theorem regarding the knot types of the periodic orbits which appear in the flow:

\begin{theorem}
    All knots appearing on the figure-eight template are positive, fibered and prime. 
\end{theorem}

Moreover, when removing the invariant figure eight knot one obtains a new flow on the figure-eight knot complement that is a hyperbolic plug, i.e. it has a hyperbolic non wandering set and the boundary torus is quasi transverse. Thus we prove the following:

\begin{theorem}
    The figure-eight knot complement has at least two non equivalent hyperbolic plugs. One with orientable foliations obtained as the DA of the suspension of the cat map, and the other with non-orientable foliations defined by the figure eight geometric model.
\end{theorem}
\section{Preliminaries}
\subsection{Three Dimensional Flows}

\begin{definition}
    An orbit which limits onto one singular point as $t \rightarrow \infty$ and limits onto a different singular point as $t \rightarrow -\infty$ is called a heteroclinic connection.
\end{definition}

\begin{definition}
    Given a flow, assume there exist heteroclinic connections between the singular points of the flow that construct a closed curve or can be extended into a closed curve, such a closed curve is called a heteroclinic knot.
\end{definition}

\begin{remark}
    In our case, we have two heteroclinic connections which can be extended into a closed heteroclinic knot passing through $\infty$.
\end{remark}

\begin{definition}
    Given a flow on a closed manifold $M$ we say that $M$ has a local cross section if there exists a closed codimension-one submanifold $N \subset M$ which transversely intersects the flow at every point of $N$.\\
    In the case where some orbits return to $N$ in a finite time, then for the subset $U \subset N$, which consists of the intersection of $N$ with the orbits which returns to $N$ in a finite time, there is a well defined first return map $r: U \rightarrow N$ which assigns to a point $p \in U$ its image $\phi ^ {T(p)}$, where $T(p)$ is the smallest $t>0$ such that $\phi (t) \in N$.
    \label{def1}
\end{definition}

\begin{remark}
    In definition \ref{def1}, the interior of $N$ must be $C^1$ but $N$ can have a boundary that may be piecewise differentiable.
\end{remark}

\subsection{Hyperbolic flows}
\begin{definition}
    An invariant set of a flow $\phi : \mathbb{R} \times M \rightarrow M$ is a set $\sigma \subset M$ such that $\phi ^t (\sigma) =\sigma$ for all $t\in \mathbb{R}$. 
\end{definition}

\begin{definition}
    An invariant set $\sigma$ for a flow $\phi _t$ on $M$ is hyperbolic if there exists a continuous $\phi _t$-invariant splitting of the tangent bundle $TM_\sigma$ into $E^s _\sigma \oplus E^u _\sigma \oplus E^c _\sigma$ with
        $$||D_{\phi _t (v)}|| \leq C e^{-\lambda t} ||v||,  \forall v\in E^s _\sigma, t>0,\\$$
        $$||D_{\phi _-t (v)}|| \leq C e^{-\lambda t} ||v||,  \forall v\in E^u _\sigma, t>0,\\$$    
        $$\frac{d\phi _t}{dt} |_{t=0} (x) \text{ spans } E^c _x,  \forall x\in \sigma\\$$
    for some fixed $C>0, \lambda>1$. The one dimensional direction $E^c _x$ is tangent to the orbit itself at each point.
\end{definition}

\begin{definition}
    Let $X \subset \sigma$ be a subset of a hyperbolic invariant set of a flow $\phi _t$ on $M$. Then the stable and unstable manifolds of $X$ in $M$ are given by
    $$W^s(X)={y \in M : lim_{t\rightarrow \infty}||\phi_t(X)-\phi_t(y)||=0},$$
    $$W^u(X)={y \in M : lim_{t\rightarrow -\infty}||\phi_t(X)-\phi_t(y)||=0},$$
\end{definition}

\begin{definition}
    Let $M$ be a closed smooth 3-manifold and let $\phi _ t : M \xrightarrow{} M$ be a flow generated by a non-singular $C^k$-vector field $X$, where $k \geq 1$. The flow is Anosov if $TM$ admits a hyperbolic splitting, for some given Riemannian metric on $M$.
\end{definition}


\subsection{Knots and templates}
In this section, most of the definitions are taken from the book \cite{ghrist2006knots}. 
Basic knot properties are taken from the book \cite{rolfsen2003knots}.

\begin{definition}
    A knot is an embedding $K:S^1 \hookrightarrow S^3$ of a 1-sphere into the 3-sphere up to isotopy. A link $L: \sqcup S^1 \hookrightarrow S^3$ is a disjoint, finite collection of knots.
\end{definition}

\begin{definition}
    A knot is called positive if it has a diagram in which all crossings are positive (see Figure \ref{fig0.2}). 
\end{definition}

\begin{remark}
    In this study, the sign of crossings is taken to be as in Figure \ref{fig0.2}, following the convention in braid theory.
\end{remark}

\begin{figure}[!h]
    \centering
    \includegraphics[width=6cm]{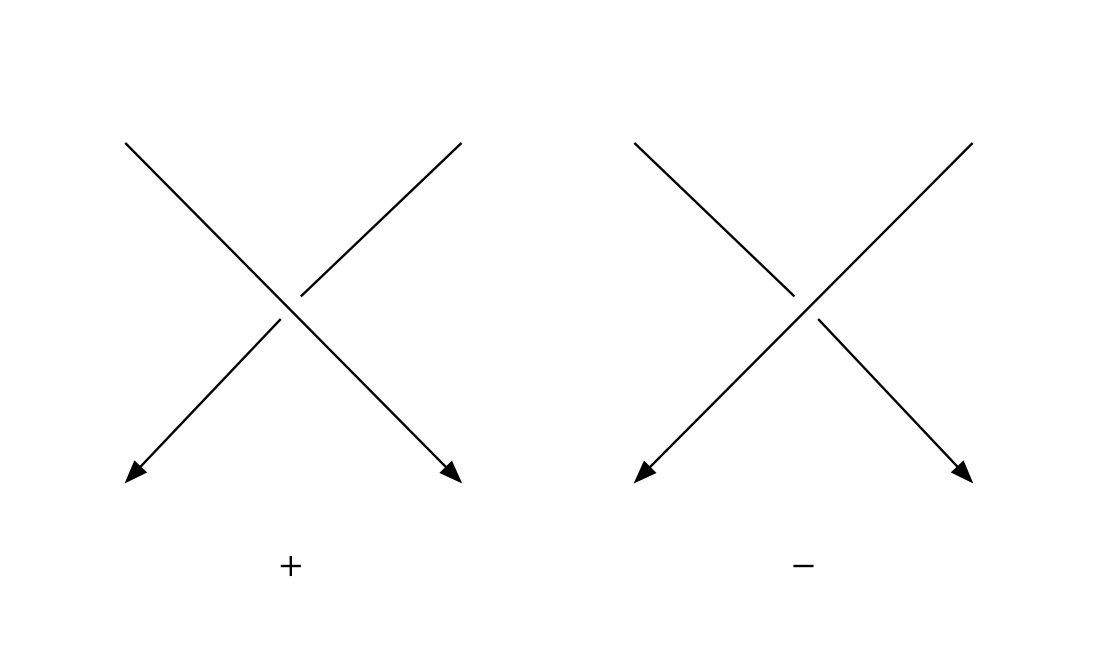}
    \caption{Sign of crossings}
    \label{fig0.2}
\end{figure}

\begin{definition}
    A flow $\phi : \mathbb{R} \times M \rightarrow M$ is called expansive if $\forall \epsilon >0$, $\exists \delta >0$ such that if $\parallel \phi_t (x) - \phi _{s(t)}(y) \parallel <\delta$, $\forall t \in \mathbb{R}$ for $x,y \in M$ and a continuous map $s: \mathbb{R} \rightarrow \mathbb{R}$ with $s(0)=0$, then $y=\phi _t (x)$ where $t<\epsilon$.
\end{definition}

\begin{definition}
    A template is a compact branched two-manifold with boundary and smooth expansive semiflow built locally from two types of charts: joining and splitting. Each chart, as illustrated in Figure \ref{fig0.1} carries a semiflow, endowing the template with an expanding semiflow, and the gluing maps between charts must respect the semiflow and act linearly on the edges which are the red lines in Figure \ref{fig0.1}.
\end{definition}

\begin{figure}[!h]
    \centering
    \includegraphics[width=10cm]{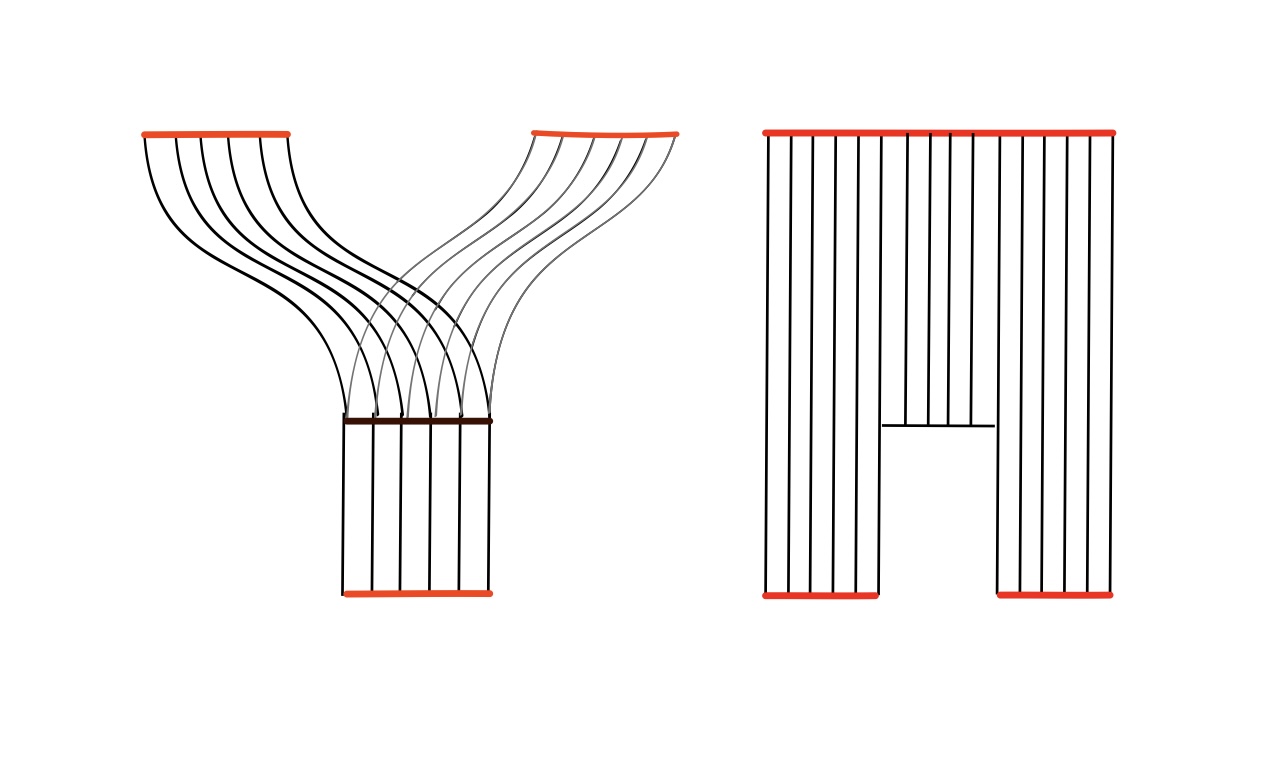}
    \caption{Joining and splitting charts}
    \label{fig0.1}
\end{figure}

\begin{definition}
    Given a flow $\phi _t$ on $M$, a point $x \in M$ is chain-recurrent for $\phi$ if, for any $\epsilon, T >0$, there exists a sequence of points ${x=x_1, x_2, ..., x_n=x}$ and real numbers ${t_1, t_2, ...,t_{n-1}}$ such that ${t_i >T}$ and $||\phi_{t_i}(x_i) - x_{i+1}||< \epsilon$ for all $1 \leq i \leq n-1$. The chain-recurrent set, $R(\phi)$, is the set of all chain-recurrent points on $M$. $R(\phi)$ is a compact set invariant under the flow.
\end{definition}

\begin{theorem}
    (the template theorem: Birman and Williams \cite{Birman1983}) Given a flow $\phi ^t$ on a three manifold $M$ having a hyperbolic chain-recurrent set, the link of periodic orbits $L_\phi$ is in bijective correspondence with the link of periodic orbits $L_\tau$ on a particular embedded template $\tau \subset M$. On any finite sublink, this is via ambient isotopy.
    \label{bw}
\end{theorem}

\begin{remark}
    In particular any Anosov flow has a template carrying its periodic orbits.
\end{remark}

\begin{definition}
    A subtemplate $S$ of a template $\tau$ is a topological subset of $\tau$ which, equipped with the restriction of the semiflow of $\tau$ to $S$, satisfies the definition of a template.
\end{definition}

\begin{definition}
    Two templates are called equivalent if they carry the same periodic orbits, i.e. if there exist a bijection between the link of the periodic orbits on one of the templates into the link of periodic orbits on the other template.
\end{definition}

\begin{definition}
    Given a template, in the projection of the template, if all the crossings of the flow lines are positive, then the template is called positive. Thus, a positive template is a template that only contains positive knots.
\end{definition}

\begin{definition} \cite{rolfsen2003knots}
    A mapping $f:E \rightarrow B$ is said to be a fibration with fiber $F$ if each point of $B$ has a neighbourhood $U$ and a 'trivializing' homeomorphism $h: f^{-1}(U) \rightarrow U \times F$ for which the following diagram commutes;
    \begin{center}
      \begin{tikzcd}
    f^{-1}(U) \arrow{r}{h} \arrow[swap]{dr}{f} & U \times F \arrow{d}{projection} \\
     & U
      \end{tikzcd}  
    \end{center}
    Each set $f^{-1}(b)$ is called a fiber and it is homeomorphic with $F$.
\end{definition}

\begin{definition} \cite{rolfsen2003knots}
    A knot or link $L$ in $S^3$ is fibered if there exists a fibration map $f: S^3 \setminus L \rightarrow S^1$ so that in a neighborhood of $L$ the fibration is $A \times S^1$ where $A$ is an annulus.
\end{definition}

\begin{definition}
    A link is decomposable if there exists a sphere which does not intersect with the link and has at least one component of the link inside and one component outside, otherwise it is indecomposable.
\end{definition}

\begin{definition}
    A link of multiplicity $\mu$ is a link with $\mu$ components, i.e. a link which is given by $L: \sqcup_{\mu} S^1 \hookrightarrow S^3$.
\end{definition}

\begin{definition}
    Given a link $L$, the genus, $g(L)$, is defined as the minimal genus over all orientable surfaces $S$ which span $L$: that is $\partial S=L$, where $\partial S$ is the orientable boundary. A spanning surface of minimal genus is called a Seifert surface.
\end{definition}

\begin{theorem}(Birman and Williams \cite{Birman1983})
    Let $L$ be an indecomposable link of multiplicity $\mu$. Suppose that $L$ has a regular projection which exhibits it as a closure of positive $n$-string braid, and that in this projection there are $c$ crossings. Then $L$ is fibered, and its genus $g$ is given by the formula:
    $$2g=c-n+2-\mu$$
    $$r=2g +\mu -1=c-n+1$$
    \label{thmfk}
\end{theorem}

\begin{definition}
    Given two oriented knots $K_1, K_2$, their connected sum, $K_1 \sharp K_2$,  is formed by placing each in disjoint embedded 3-balls, $B_1, B_2$, such that some arc of $K_1 (K_2)$ lies on the boundary of $B_1$ ($B_2$ respectively). Then delete the interior of each arc and identify the boundaries of the arcs via an orientation preserving homeomorphism.  
\end{definition}

\begin{remark}
    In the definition of connected sum, the operation is well-defined since the choice of arcs does not effect the connected sum. This operation is commutative and associative, but it is not a group operation since it does not have an inverse. 
\end{remark}

\begin{definition}
    If a knot can be decomposed into the connected sum of two or more non-trivial knots, it is said to be composite, otherwise it is prime.
\end{definition}

\begin{definition}
    A knot $K$ is called visually prime if it is given by a diagram that is a prime diagram, i.e. for any simple closed curve $\eta$ on the projection plane, intersecting the knot projection at two points, one of the components of $K$ minus the intersection points is a trivial arc in the disc bounded by $\eta$.
\end{definition}

\begin{theorem} (Ozawa \cite{ozawa2002closed})
    Non-trivial positive knots or links are prime if their positive diagrams are connected and prime.
    \label{thmpk}
\end{theorem}

\section{the first return map}
By numerical studies \cite{creaser2015alpha}, at the second T-point, there exists an invariant figure-eight heteroclinic knot (see Figure \ref{fig1.01}). In this study, we assume that there exists a T-point with a figure-eight invariant heteroclinic knot. We are interested in studying the flow in this case where we aim to classify the periodic orbits which appear in the flow. To do so, we first need to understand how the first return map looks like for a given cross section.\\

\begin{figure}[!h]
    \centering
    \includegraphics[width=7cm]{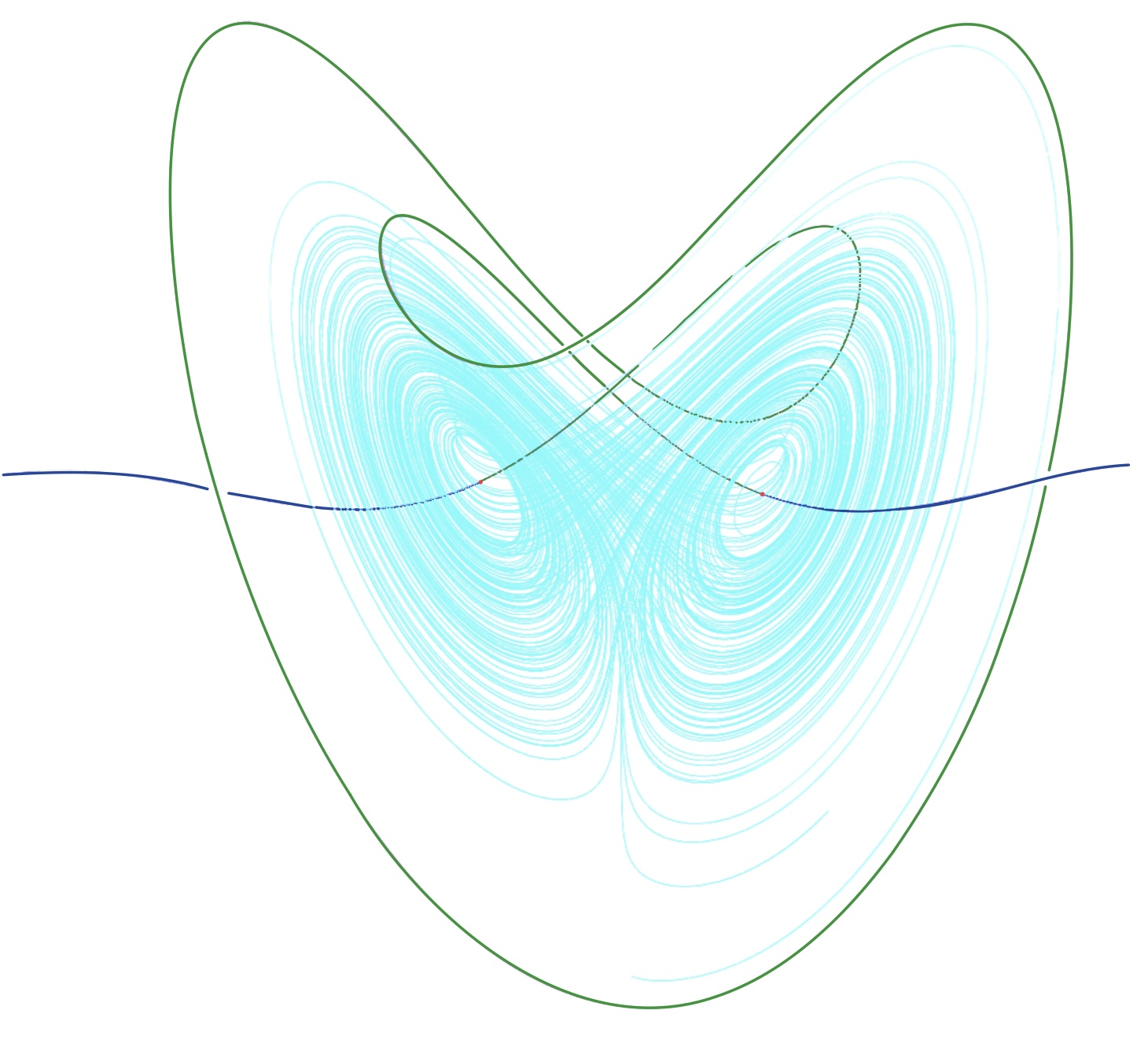}
    \caption{Figure eight knot obtained by numerical approach}
    \label{fig1.01}
\end{figure}

By \cite{Pinsky2023}, there exists a two dimensional cross section $R$ with boundary, which includes the singular points $p^\pm$ on its boundary, and its interior is transverse to the flow except for orbits which limit into the singular points. In this section we will describe the first return map of $R$.

\begin{figure}[!h]
    \centering
    \includegraphics[width=7cm]{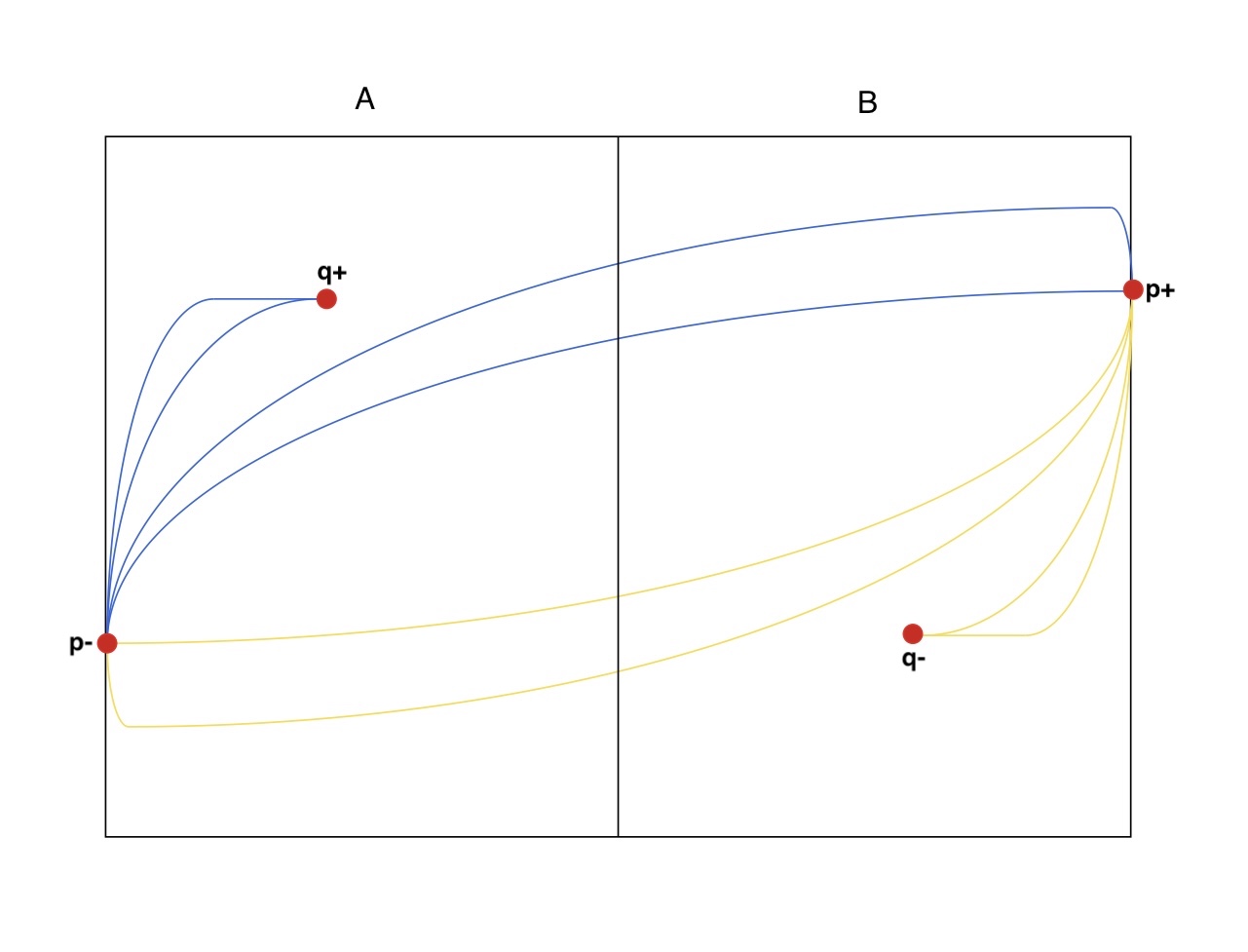}
    \caption{The first return map}
    \label{fig1}
\end{figure}

    The stable manifold of the origin separates the cross section $R$ into two parts which we denote by $A$ and $B$. We will prove bellow some properties of the first return map for the given cross section $R$, and we will show that the first return map of $R$ must have two symmetric "hooks", for an example see Figure \ref{fig1}; the image of part $A$ under the first return map contains the yellow "hook", and the image of part $b$ contains the blue one.\\
    We will prove below that in order to have an invariant figure-eight knot, the heteroclinic orbits which flow from the origin to the singular points $p^\pm$ must pass through the rectangle twice, therefore, there is a point $q^+$ that is mapped to the singular point $p^+$ and a point $q^-$ that is mapped to the singular point $p^-$, see Figure \ref{fig1}.\\

\begin{theorem}
    Assume that the second T-point exists, and assume that at this T-point there exists an invariant figure-eight heteroclinic knot. Then the symmetric heteroclinic orbits which flow from the origin to the singular points $p^\pm$, must pass through the cross section $R$ at least once before ending at the points $p^\pm$.
    \label{thm1}
\end{theorem}

\begin{proof}
    First let us notice that by the symmetry of the system $(x,y,z) \rightarrow (-x,-y,z)$, whenever we have a a heteroclinic orbit which flows from the origin to the singular point $p^+$, there must be a symmetric heteroclinic orbit which flows from the origin to the symmetric singular point $p^-$, we will denote this heteroclinic orbits by $\alpha ^\pm$.
    As we have mentioned before, these connections can be extended into $\infty$. By this extension we get an invariant heteroclinic knot which connects all the singular points of the system, $p^\pm$, the origin, and $\infty$, this invariant heteroclinic knot will be denoted by $\gamma$.
    Moreover, if we take the projection of $\gamma$ on the plane $x=y$, the part which connects the points $p^\pm$ through $\infty$ does not have any crossings (with itself).
    Therefore, we can identify the type of the invariant heteroclinic knot, $\gamma$, by analysing only the heteroclinic orbits $\alpha ^\pm$.\\
    In addition, let us notice that the $z$-variable increases or decreases monotonely depending on the range of the points with respect to the hyperbolic paraboloid $\mathcal{P}=\{yx=\beta z\}$. The $z$-variable monotonely increases outside $\mathcal{P}$, $\{xy>\beta z\}$, and monotonely decreases inside $\mathcal{P}$, i.e. $\{xy<\beta z\}$.\\
    Assume for a contradiction that the orbits $\alpha ^\pm$ does not pass through the cross section $R$ before ending at the points $p^\pm$, in other words, the orbits $\alpha ^\pm$ intersects with the cross section only at the points $p^\pm$. 
    Then, the $z$-variable of the orbits $\alpha^\pm$ monotonely increases until the orbits cross through $\mathcal{P}$, then the $z$-variable decreases monotonely until the orbits end at the singular points $p^\pm$.
    Therefore, the only way to get the figure-eight knot whenever we project on the plane $x=y$, the orbits $\alpha ^\pm$ must rotate around each other before ending at the points $p^\pm$ (see Figure \ref{fig1.4}).
    In this case, since the flow is continuous, there exist flow lines which connects part A and B with the heteroclinic orbits $\alpha ^\pm$ respectively. This implies that the images of parts A and B of $R$ under the first return map intersect (see Figure \ref{fig1.2}), which is a contradiction for the uniqueness of the solution for ordinary differential equations.
\end{proof}

\begin{figure}[!h]
    \centering
    \includegraphics[width=7cm]{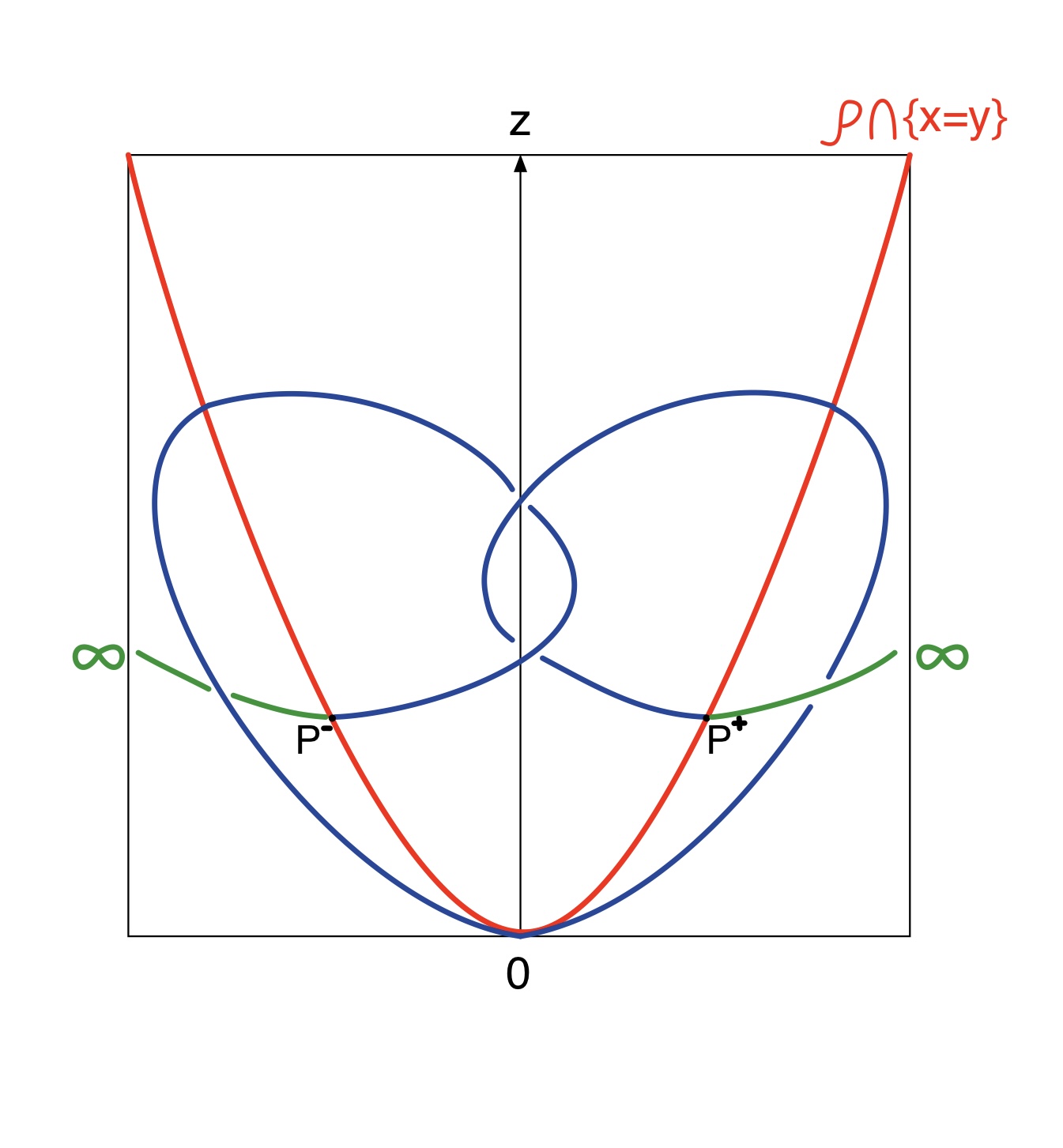}
    \caption{The projection of the heteroclinic orbits $\alpha ^ \pm$ on the plane ${x=y}$.}
    \label{fig1.4}
\end{figure}

\begin{figure}[!h]
    \centering
    \includegraphics[width=7cm]{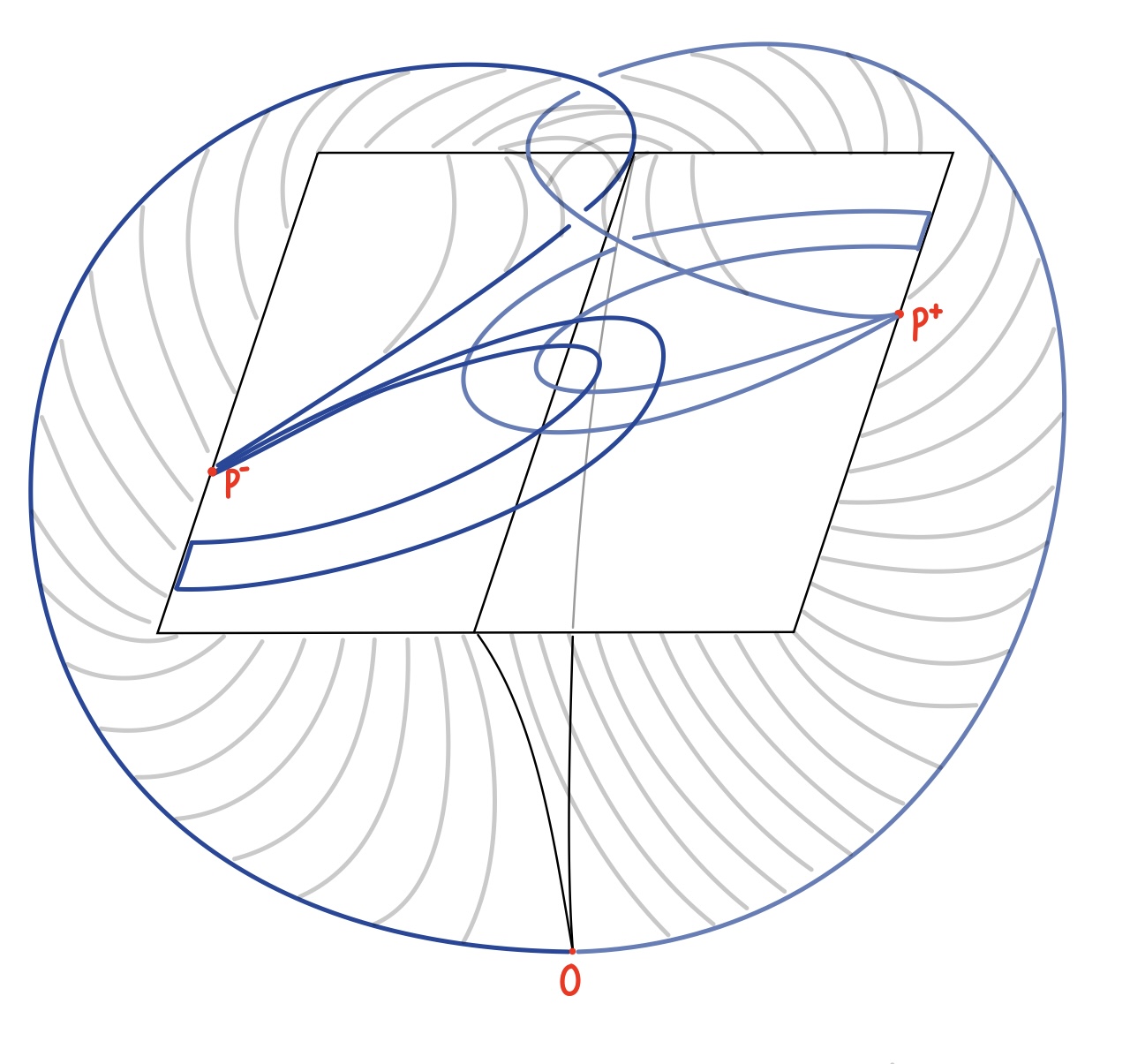}
    \caption{If the seperatrix does not pass $R$ then the flow lines intersect - a contradiction.}
    \label{fig1.2}
\end{figure}

\begin{corollary}
    There exist points $q^\pm$ in the rectangle $R$ which are mapped to the singular points $p^\pm$.
\end{corollary}

\begin{proof}
    By Theorem \ref{thm1}, the symmetric heteroclinic orbits which flow from the origin to the singular points $p^\pm$ pass through the cross section at least once before ending at the singular point $p^\pm$, therefore, there exist two points, $q^\pm$, in the cross section $R$, which flow to the singular points $p^\pm$.
\end{proof}

\begin{theorem}
    Under the assumptions of Theorem \ref{thm1}, and assume that the heteroclinic orbits $\alpha ^\pm$ pass through the cross section exactly once before ending at the points $p^\pm$. Then, the point $q^+ \in A$, and the point $q^- \in B$.
    \label{thm2}
\end{theorem}


\begin{proof}
    Assume for a contradiction that $q^+ \in B$ which implies by the symmetry that $q^- \in A$. By the monotonicity of the $z$-variable, and the fact that the orbits which flow from the origin to the points $q^\pm$ can not rotate around each other, i.e. these orbits will not have any crossings with each other when projecting on the plane $x=y$, as we proved in Theorem \ref{thm1}, we get that the heteroclinic knot $\gamma$ is a connected sum of two knots. These knots are the obits $\alpha ^\pm$ with the extension to $\infty$, where we close each of them by an orbit that connects the origin with the $\infty$.\\
    By the symmetry of the system, if one of the knots is the unknot, then also the other one is the unknot, which implies that the heteroclinic knot is the unknot. Otherwise, the heteroclinic knot can be decomposed into a connected sum of two non-trivial knots, therefore, in this case the heteroclinic knot $\gamma$ is not prime, hence not the figure-eight knot -  a contradiction.
\end{proof}

\begin{theorem}
    Under the assumptions of Theorem \ref{thm2}, if we divide the cross section $R$ into three parts as in Figure \ref{fig4.1}, then the point $q^+$ is in part $1$ and the point $q^-$ is in part $3$.
    \label{thm3}
\end{theorem}

\begin{figure}[!h]
    \centering
    \includegraphics[width=7cm]{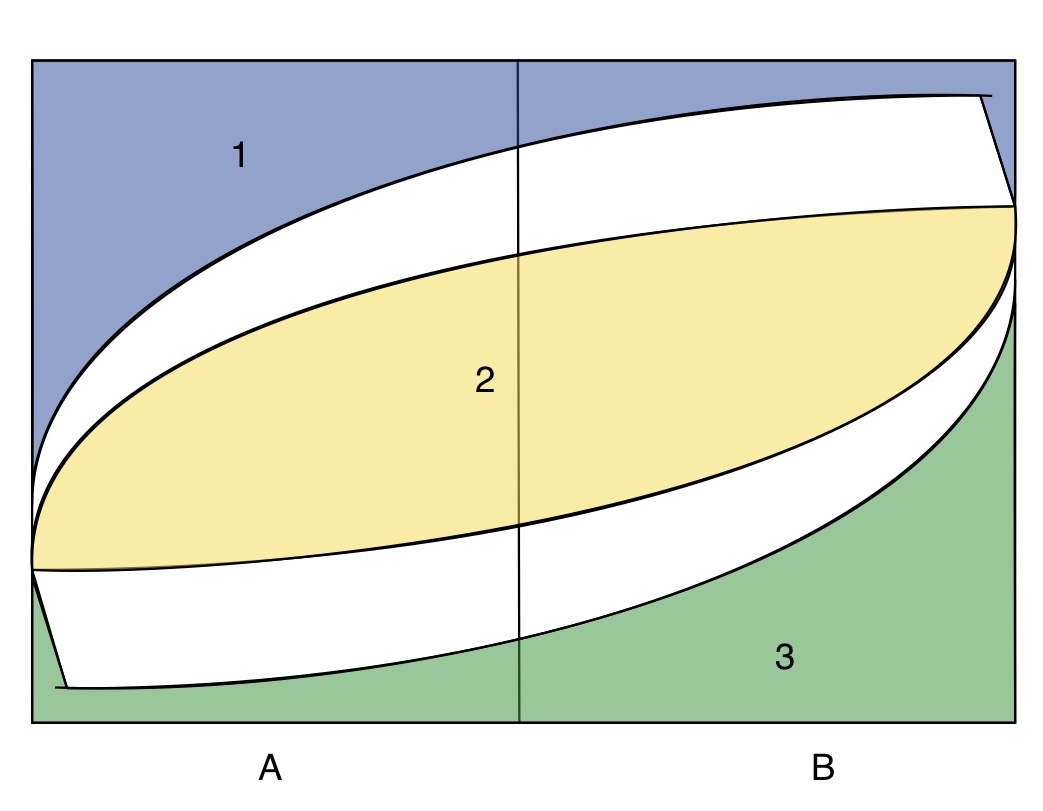}
    \caption{We divide the rectangle $R$ into three parts 1 - blue, 2 - yellow and 3 - green}
    \label{fig4.1}
\end{figure}

\begin{proof}
    In Theorem \ref{thm2}, we proved that $q^- \in B$ and $q^+ \in A$, therefore, if we look at part $A$ of $R$, the points $q^+,p^- \in A$, hence the points $p^\pm$ are in the image of $A$ under the first return map. Using the fact that the flow is continuous implies that the image of $A$ under the first return map, contains a path connecting the points $p^+, p^-$, and by the symmetry, we get the same for the image of part $B$ of the cross section $R$.\\
    There are three possible options which follows from the described above. 
    The first option is that both of the points $q^\pm$ lie on part $3$ of $R$, in this case, as we can see in Figure \ref{fig4.2}(a) the hetroclinic connection that we get is $9_{46}$ knot.\\
    The second option is that the point $q^+$ is in part $3$ and the point $q^-$ is in part $1$, in this case, the hetroclinic connection that we get is also $9_{46}$ knot as we can see in Figure \ref{fig4.2}(b).\\
    The last option is the only one which for the hetroclinic connection gives the figure-eight knot as we can see in Figure \ref{fig4.2}(c).
\end{proof}

\begin{figure}[!h]
    \centering
    \includegraphics[width=10cm]{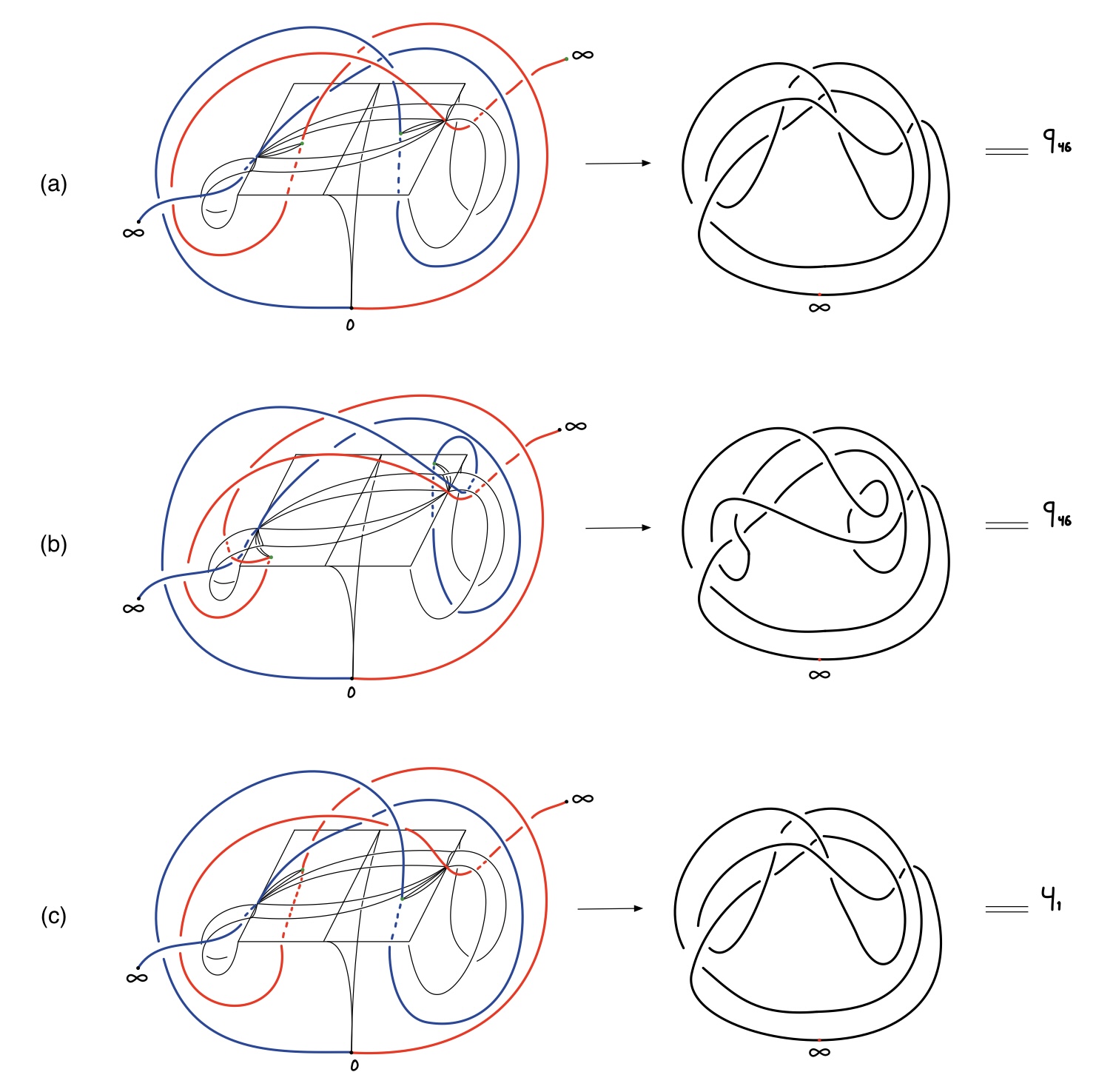}
    \caption{(a) At the first case we get the $9_{46}$ knot. (b) At the second case we get the $9_{46}$ knot. (c) At the last case we get figure-eight knot}
    \label{fig4.2}
\end{figure}

\begin{theorem}
Under the assumptions of Theorem \ref{thm2}, the first return map in the case of the figure-eight knot includes two "hooks", one contains a path that connects the points $p^-, p^+, q^-$ and is contained in the image of $A$, and the second contains a path that connects the points $p^+, p^-, q^+$ and is contained in the image of $B$, see for example Figure \ref{fig1}.    
\end{theorem}

\begin{proof}
    By Theorem \ref{thm1} and Theorem \ref{thm3}, there exist a heteroclinic orbit which flows from the origin to the point $q^- \in A$. Since the stable manifold of the origin intersects with A, and the flow is continuous, we conclude that points in $A$ near the stable manifold of the origin, are mapped to points near $q^-$. And as we have mentioned before, $p^-\in A$ is a singular point. This implies that the image of $A$ under the first return map, contains a path connecting the points $p^-, q^-$, see Figure \ref{fig4.3}.
    In addition, the point $q^+ \in A$, is mapped to $p^+ \in B$, which means that the point $p^+$ is also in the image of $A$.
    Therefore, there exists some "hook" in the image of $A$ which connects the points $p^\pm, q^-$, see for example the yellow "hook" in Figure \ref{fig1}. By the symmetry of the system, there exists also a symmetric "hook" in the image of $B$ which connects the points $p^\pm, q^+$.
\end{proof}

\begin{figure}[!h]
    \centering
    \includegraphics[width=7cm]{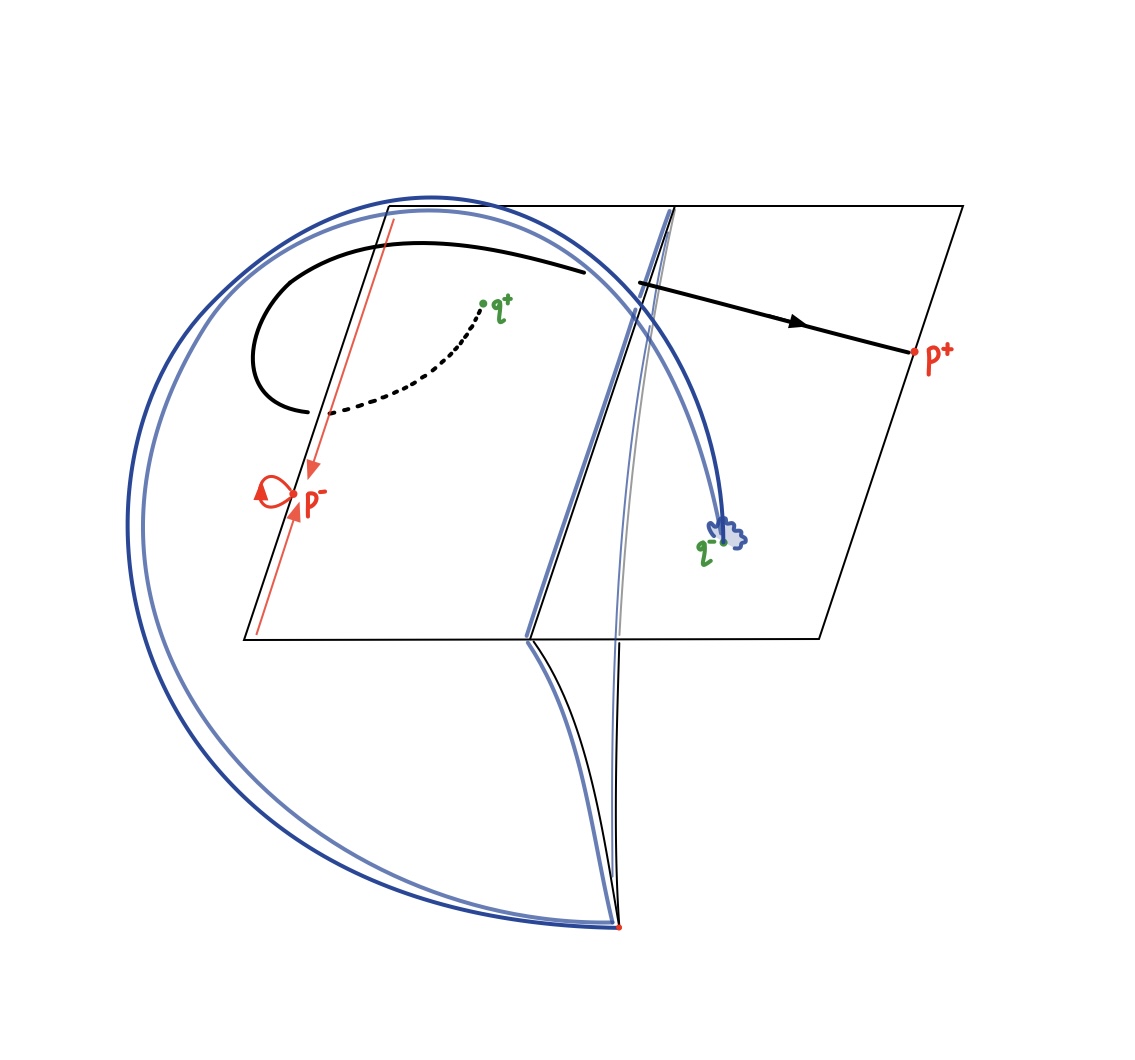}
    \caption{The properties of the heteroclinic orbits $\alpha^\pm$ implies the existence of two paths in the image of $R$ under the first return map, these paths connects the points $(p^\pm, q^+)$, $(p^\pm, q^-)$ .}
    \label{fig4.3}
\end{figure}

\section{Hyperbolic flow on the figure-eight complement}
\subsection{A New Flow}
Now we want to define a new flow, this flow was inspired by the first return map for Lorenz system at the figure-eight knot case. We will show that this flow is actually a hyperbolic plug with non-orientable foliations, which is defined on the figure-eight knot complement.\\
The new flow is defined on a rectangle $R$ which is divided into four open rectangles $A, B, C$ and $D$. The image of each under the flow which we denote by $\Phi$ is also a rectangle where its width expands by $\lambda= 1+\sqrt{3}$ and its height contracts by $\lambda$.\\
We blow up the rectangle $R$ and we get a three dimensional flow on a box were the lower base is divided into four open rectangles $A, B, C$, $D$ and the upper base is the image of the lower base under the flow defined above (see Figure \ref{fig6.2}).

\begin{figure}[!h]
    \centering
    \includegraphics[width=11cm]{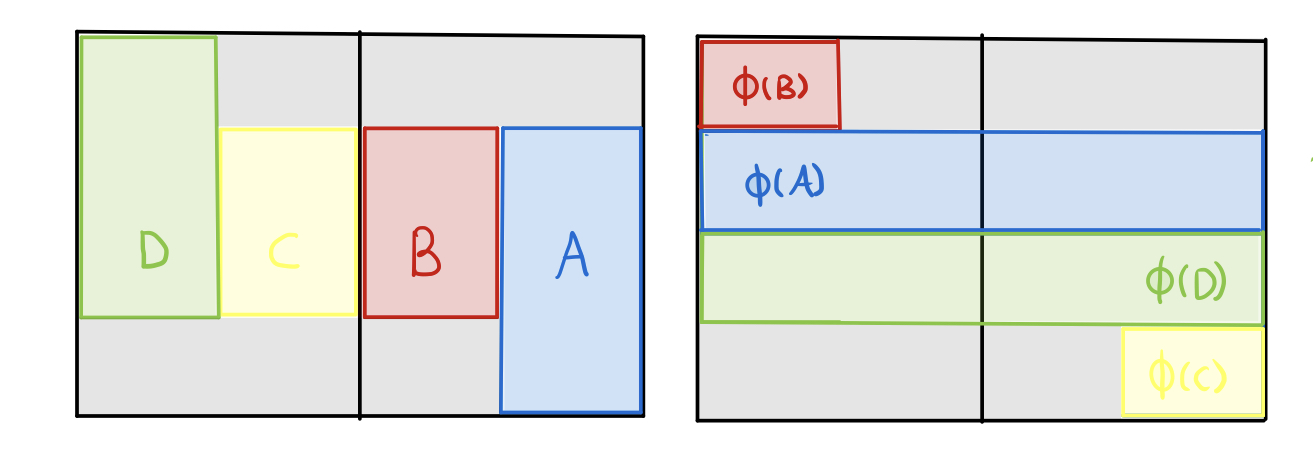}
    \caption{The Flow}
    \label{fig6.1}
\end{figure}

We want to show that this new flow is a hyperbolic plug, our flow is defined on an open set, therefore, first we need to define hyperbolic plug for a manifold with boundary.

\subsection{Hyperbolic Plugs}

\begin{figure}[!h]
    \centering
    \includegraphics[width=10cm]{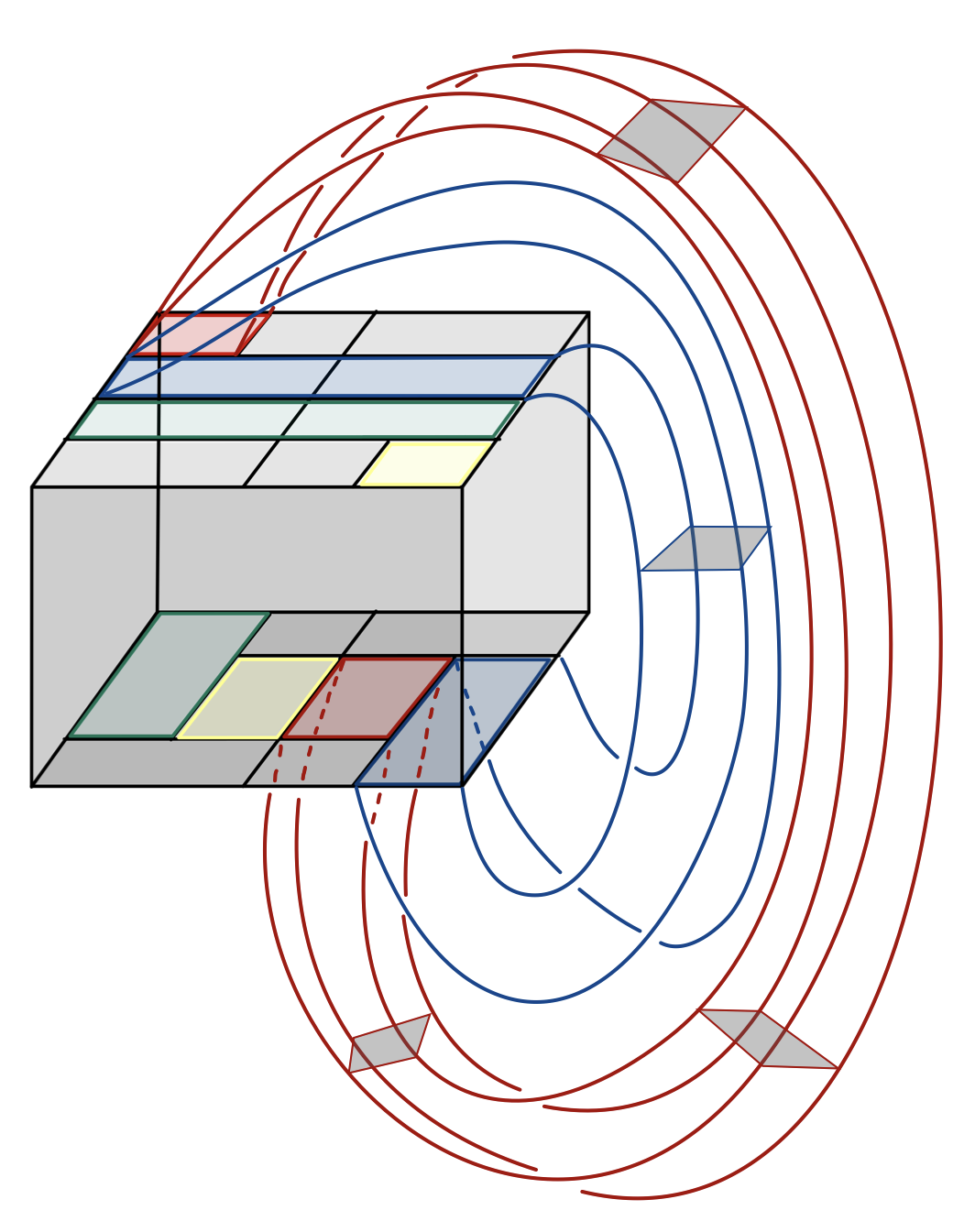}
    \caption{Hyperbolic flow}
    \label{fig6.2}
\end{figure}

\begin{definition} (quazi-transverse) The flow on the manifold is quazi-transverse to the flow on the boundary if they are transverse except for a finite number of tangent orbits. 
\end{definition} 

\begin{definition} (Hyperbolic plug)
Let $M$ be a smooth 3-manifold with boundary and let $\phi _ t : M \xrightarrow{} M$ be a flow generated by a non-singular $C^k$-vector field $X$, where $k \geq 1$. The flow is a hyperbolic plug if $TM$ admits a hyperbolic splitting, for some given Riemannian metric on $M$ and the flow on the boundary is transversal or quazi-transversal to the flow on the manifold.
\end{definition}

In the following section we will prove that the new flow is a hyperbolic plug and we will show that there exist at least two different hyperbolic flows on the figure-eight knot complement, before we do so we will introduce another known flow.

\subsection{The DA of the suspension of the cat map}
\begin{definition}
  Consider the two dimensional torus as the quotient space $\mathcal{R}^2/\mathcal{Z}^2$, Arnold's cat map $\Gamma : \mathcal{T}^2 \rightarrow \mathcal{T}^2$ is defined by $\Gamma ((x,y))= \begin{pmatrix}
        1&1\\
        1&2
    \end{pmatrix}
    \begin{pmatrix}
        x\\
        y
    \end{pmatrix} mod1$.\\
    The suspension of the cat map is a flow defined on the three manifold ${\mathcal{T}^2 \times I} / {(x,0) \sim (\Gamma(x),1)}$.
\end{definition}

The suspension of the cat map is a hyperbolic plug, it has one stable direction and one unstable direction. The DA or derived from Anosov construction is a procedure we do on the flow which destroys the stable direction in a small neighborhood of an orbit. The procedure is done in the following way: take a small tubular neighborhood which includes the orbit corresponding to the origin, and modify the flow inside the neighborhood in a way that turns the stable direction of the flow inside the neighborhood into an unstable one. 
This turns the orbit into a repelling orbit, and then one may remove a small neighborhood of the orbit such that the flow is transverse to its boundary.
The DA of the suspension of the cat map is a hyperbolic plug which is defined on figure-eight complement and has orientable foliations.
\newpage

\subsection{A non-orientable hyperbolic plug on the figure-eight knot complement}
\begin{theorem}
    The figure-eight knot complement has at least two non equivalent hyperbolic plugs. One with orientable foliations obtained as the DA of the suspension of the cat map, and the other with non-orientable foliations defined above.
\end{theorem}

\begin{proof}
    First, we will prove that the flow we defined is a hyperbolic plug. As we can see in Figure \ref{fig6.2}, on each rectangle we can define a foliation with transversal directions were one contracts by $\lambda$ and the other expands by $\lambda$.
    Hence, the flow admits a hyperbolic splitting. Moreover, on the boundary we have orbits which flow from $\infty$ and we have two orbits which are tangent to the flow, which means that the flow on the manifold is quazi-transverse to the flow on the boundary, therefore, the defined flow is a hyperbolic plug.\\
    Since the flow is a hyperbolic plug, it has a template carrying its periodic orbits. We can find such a template by collapsing the contracting direction, and we get the following template in Figure \ref{fig52}. This template is non-orientable since it has a non-orientable subtemplate as we can see in Figure \ref{fig6.3}, which implies that the new flow is aa hyperbolic plug with non-orientable foliations.\\
    \begin{figure}[h]
        \centering
        \includegraphics[width=7cm]{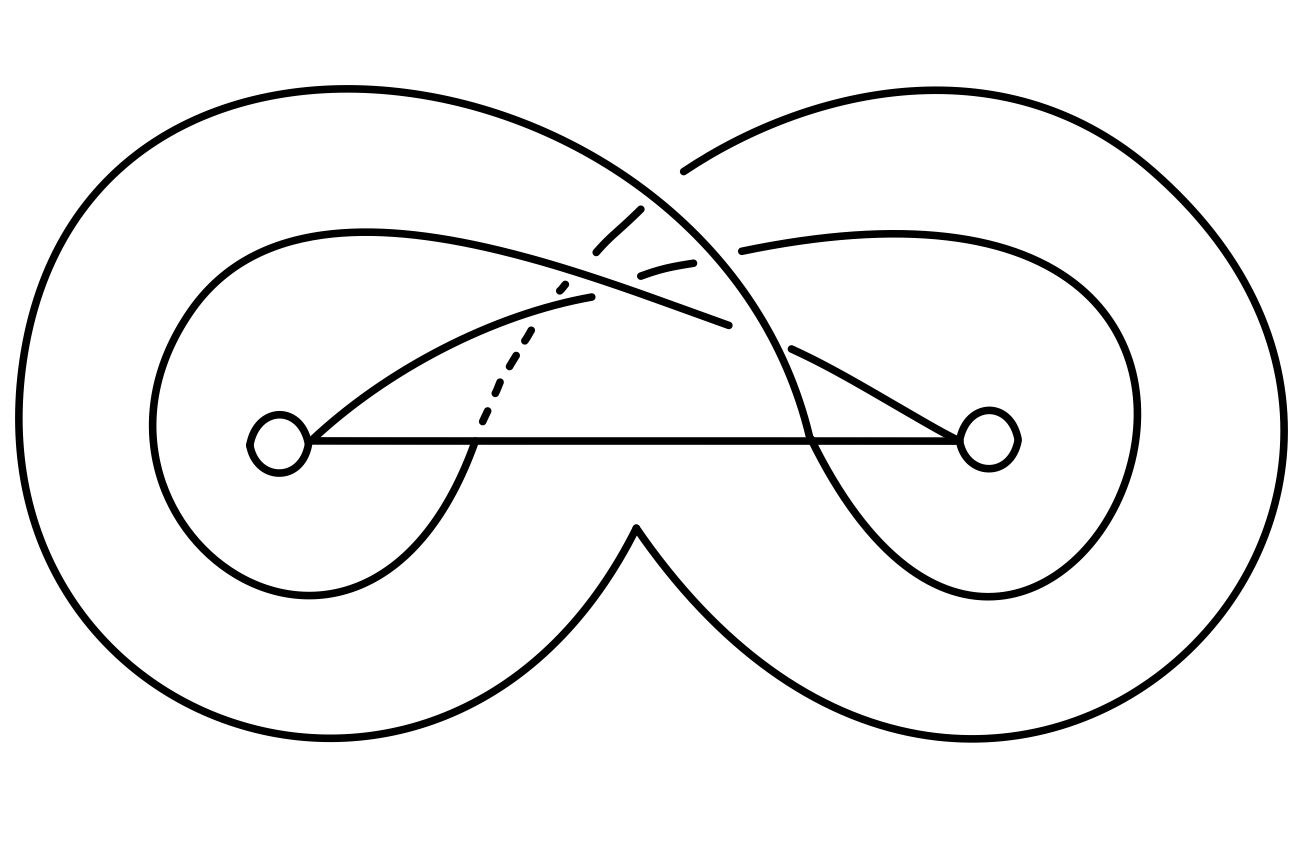}
        \caption{The template which carries the periodic orbits of the new hyperbolic plug}
        \label{fig52}
    \end{figure}
    We have proved above that the flow is a hyperbolic plug, the only thing left is to show that the flow is defined on the figure-eight knot complement. To do so, we have shown that the figure-eight knot is homeomorphic to a neighborhood of the template as can be seen in Figure \ref{fig6.4}.
    \begin{figure}
        \centering
        \includegraphics[width=10cm]{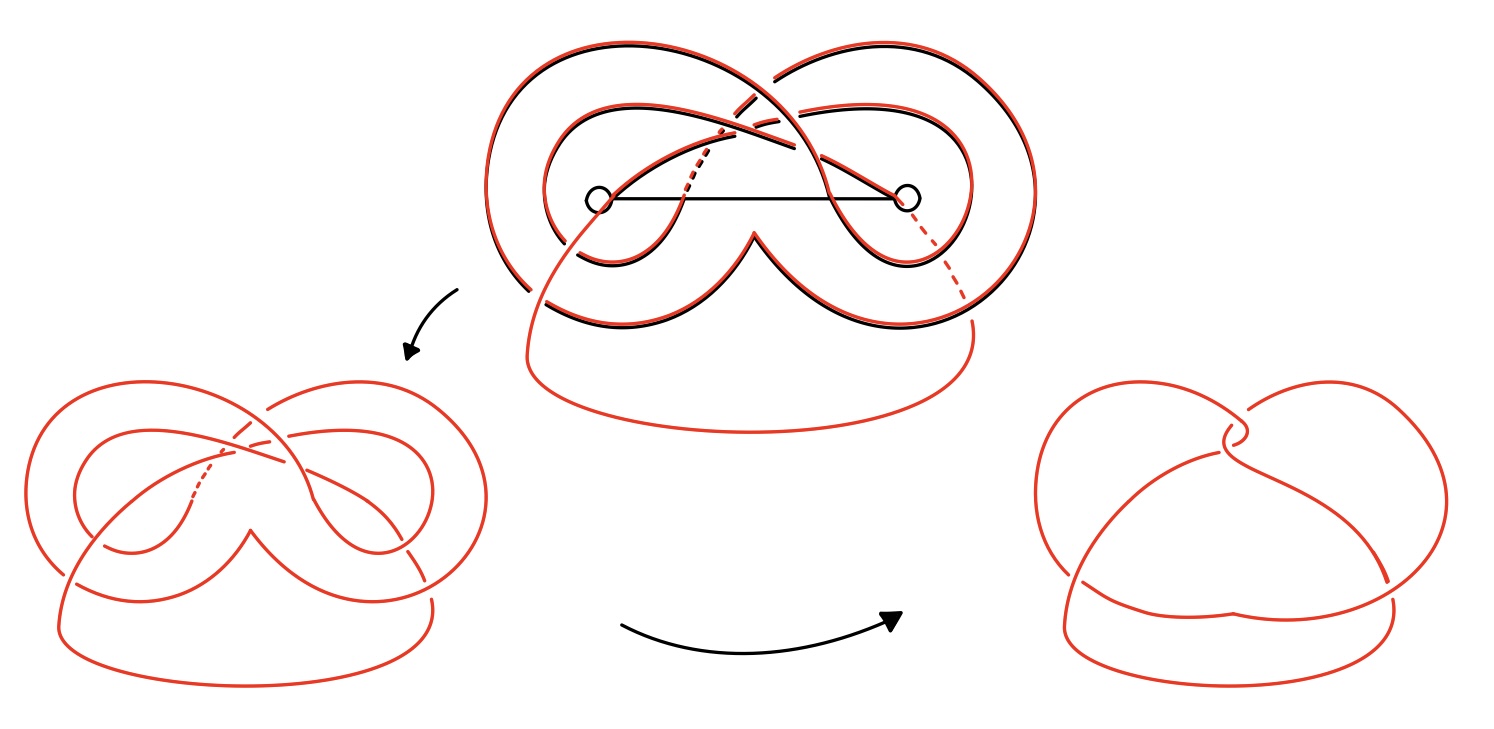}
        \caption{The figure-eight knot complement is homeomorphic to a neighborhood of the template}
        \label{fig6.4}
    \end{figure}
\end{proof}

\begin{figure}
    \centering
    \includegraphics[width=7cm]{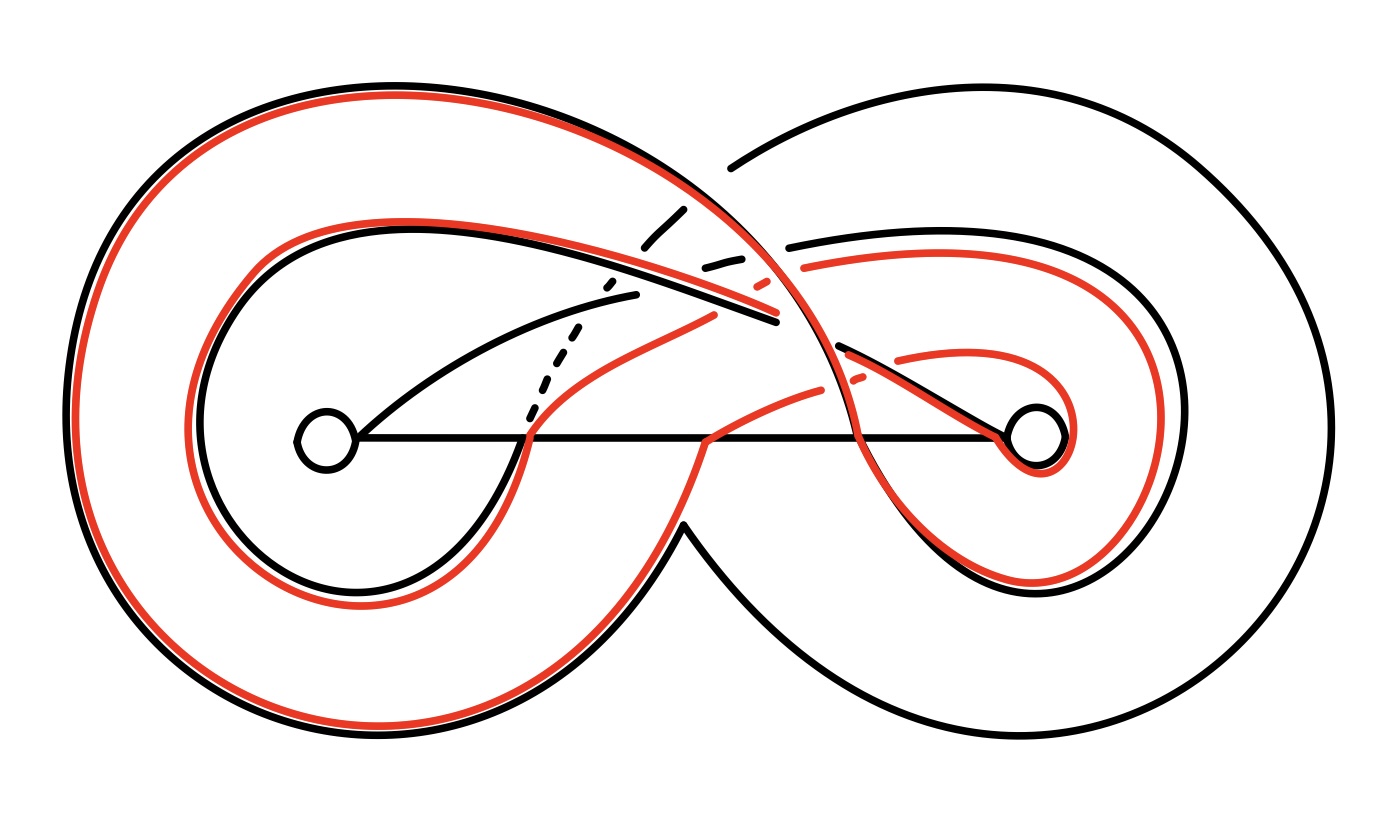}
    \caption{The template is non-orientable}
    \label{fig6.3}
\end{figure}

\section{The Template}
In this section we will discuss the properties of the template we found in the previous section, by Theorem \ref{bw}, the link of periodic orbits of the hyperbolic flow which we found, are in bijective correspondence with the link of the periodic orbits on the template.
It is easy to see that the Lorenz template \ref{fig5.3} is a subtemplate of the figure-eight template \ref{fig6} by considering only the two inner bands.
\begin{figure}[!h]
    \centering
    \includegraphics[width=8cm]{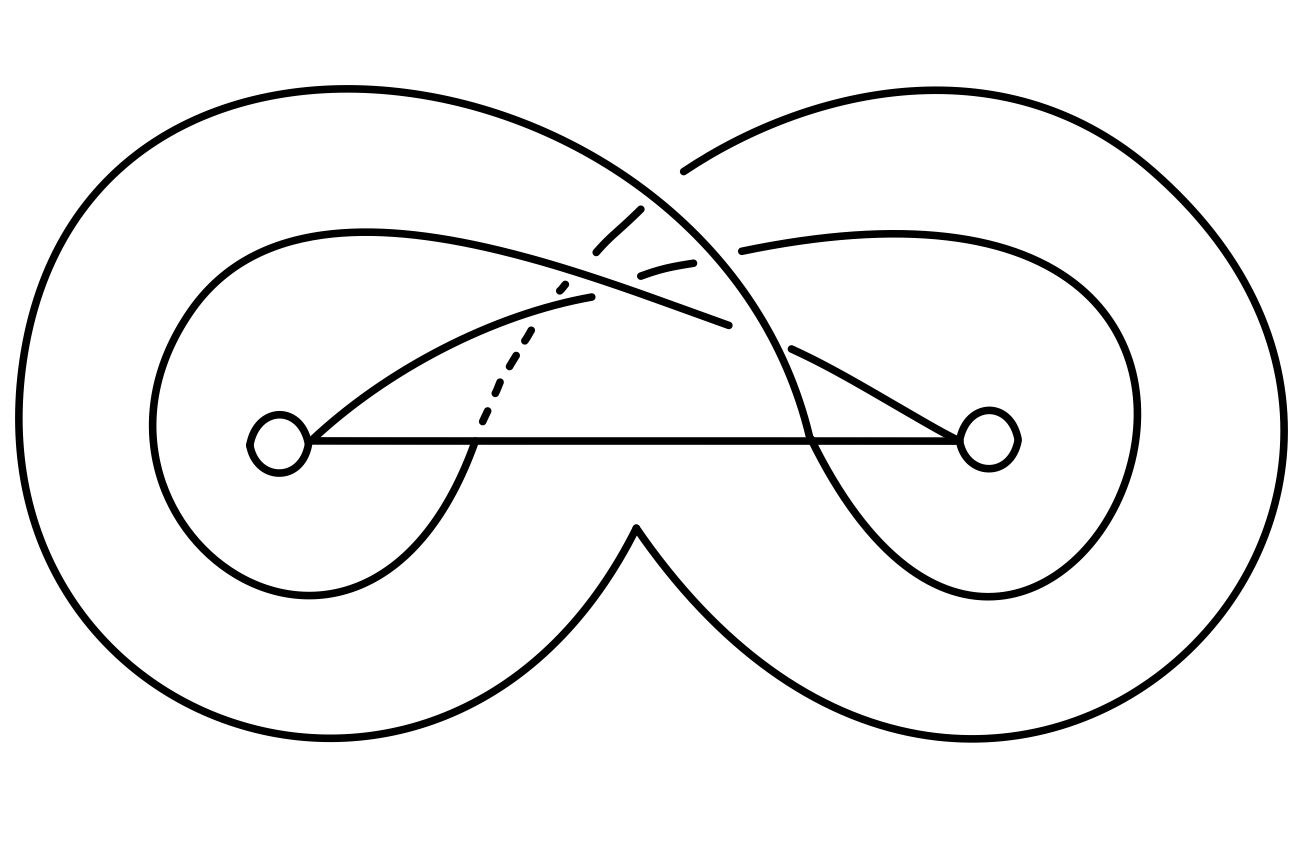}
    \caption{The Figure-eight Template}
    \label{fig6}
\end{figure}

\begin{theorem}
    All knots appearing on the figure-eight template are positive, fibered and prime.
\end{theorem}

\begin{proof}
     The figure-eight template is equivalent to the template given in Figure \ref{fig5.4}, the equivalence is obtained by separating the two outermost bands of the template from the branch line, extend them one step further on the template, and glue them back to the branch line. Then by isotopy of the bands in $\mathbb{S}^3$, we get the equivalent template given in Figure \ref{fig5.4}. The equivalent template is positive, this implies that the template of the hyperbolic plug is positive, which means that all the knots that exist on the template are positive.

\begin{figure}[!h]
    \centering
    \includegraphics[width=8cm]{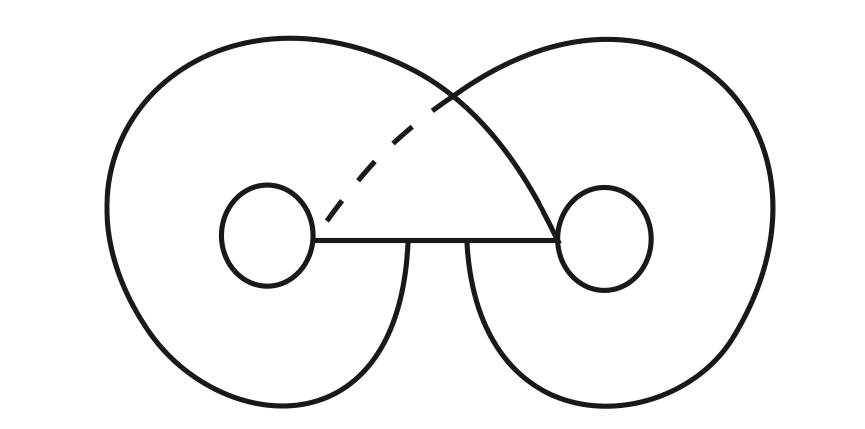}
    \caption{Lorenz Template}
    \label{fig5.3}
\end{figure}

\begin{figure}[!h]
    \centering
    \includegraphics[width=9cm]{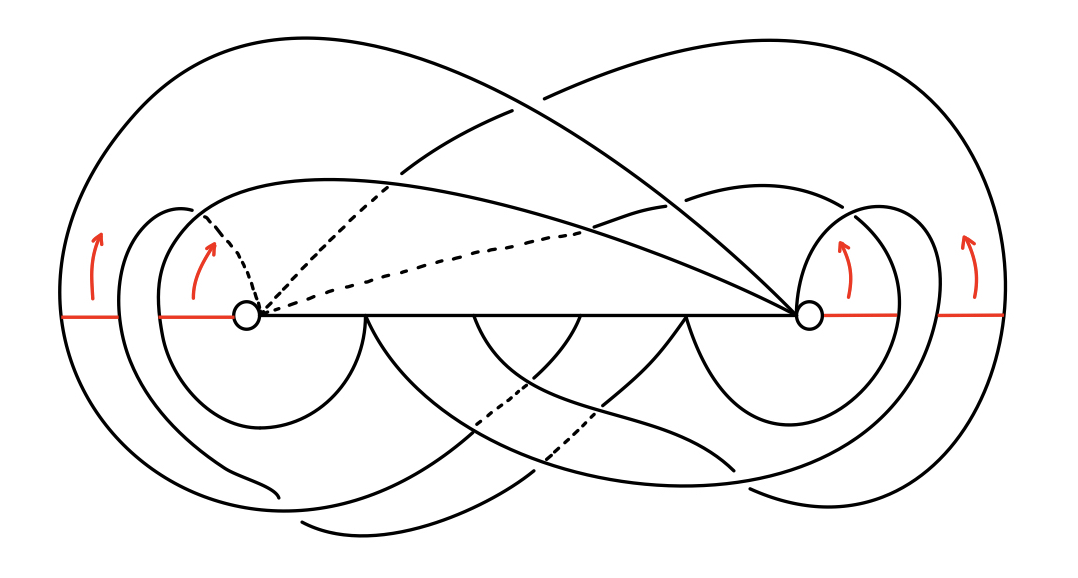}
    \caption{Equivalent Template}
    \label{fig5.4}
\end{figure}

By Theorem \ref{thmfk} closed positive braids are fibered links, therefore, by cutting the template along the red lines (see Figure \ref{fig5.5}), we can see that every link on the template can be represented as a closed positive braid, this implies that all the links on our template are fibered.

Let $K$ be a knot on the equivalent template, and let us assume that the diagram of the projection of $K$ can be represented as a connected sum of two knots $K=K_1 \sharp K_2$, then if we cut the template vertically in the middle, one of the knots $K_1, K_2$ lives only on one half of the template, which implies that this knot is the unknot. Therefore, all the knots on the template are visually prime, hence by Theorem \ref{thmpk}, all the knots on the template are prime.

\begin{figure}
    \centering
    \includegraphics[width=9cm]{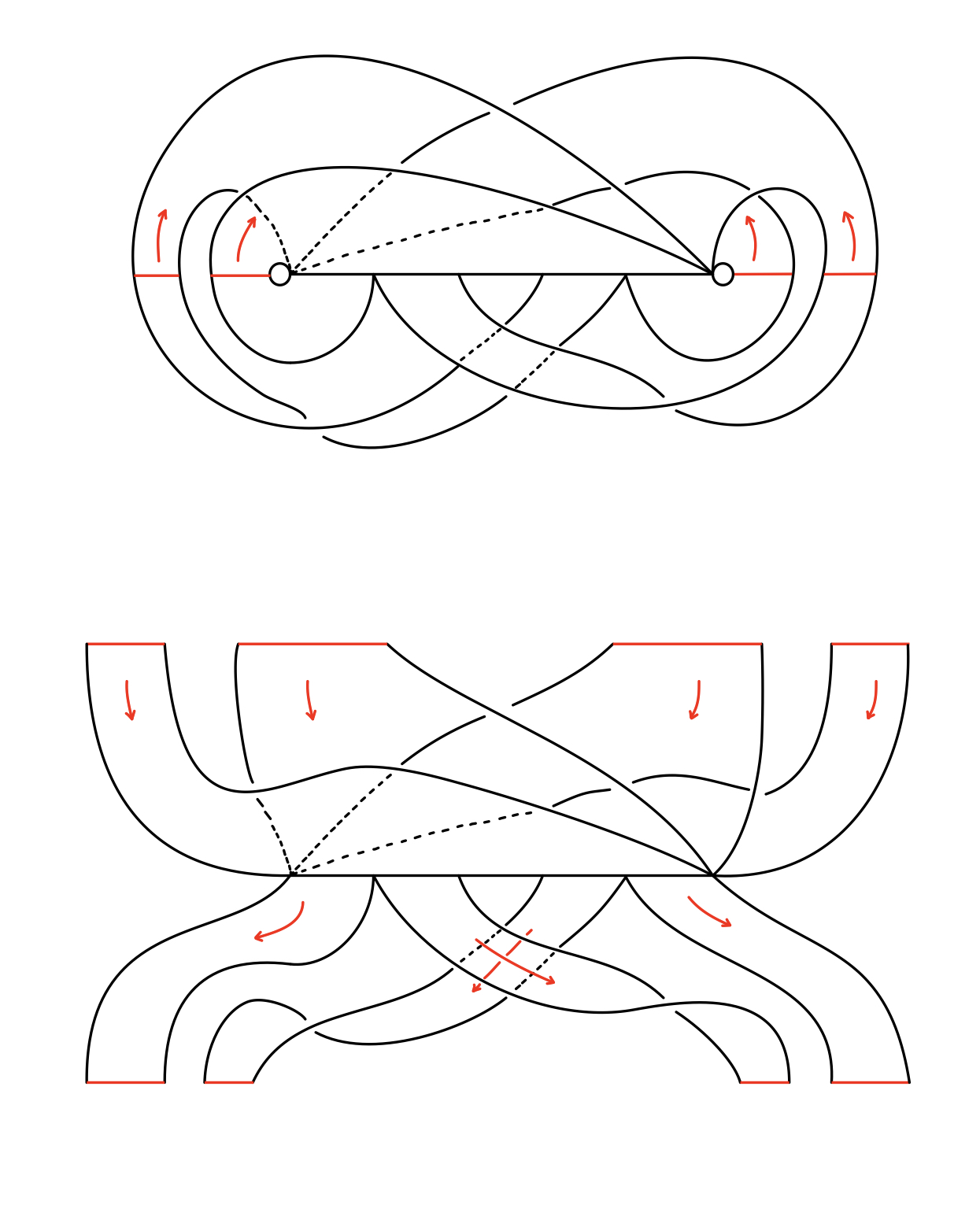}
    \caption{Cutting the template along the red lines gives a way of representing every link on the template as a closed braid}
    \label{fig5.5}
\end{figure} 
\end{proof}

\begin{conjecture}
Our template has more knots than the Lorenz template.
\end{conjecture}

\begin{conjecture}
For any knot appearing as a periodic orbit on the template there exist a periodic orbit for the original equation at the figure-eight parameter with the same knot type.
\end{conjecture}

\bibliographystyle{plain}
\bibliography{references/References.bib}

\begin{thebibliography}{1}

\bibitem{Birman1983}
Joan~S. Birman and R.~F. Williams.
\newblock {Knotted periodic orbits in dynamical systems-I: Lorenz's equation}.
\newblock {\em Topology}, 22(1):47--82, 1983.

\bibitem{creaser2015alpha}
Jennifer~L Creaser, Bernd Krauskopf, and Hinke~M Osinga.
\newblock $\alpha$-flips and t-points in the lorenz system.
\newblock {\em Nonlinearity}, 28(3):R39, 2015.

\bibitem{ghrist2006knots}
Robert~W Ghrist, Philip~J Holmes, and Michael~C Sullivan.
\newblock {\em Knots and links in three-dimensional flows}.
\newblock Springer, 2006.

\bibitem{lorenz1963deterministic}
Edward~N Lorenz.
\newblock Deterministic nonperiodic flow.
\newblock {\em Journal of atmospheric sciences}, 20(2):130--141, 1963.

\bibitem{ozawa2002closed}
Makoto Ozawa.
\newblock Closed incompressible surfaces in the complements of positive knots.
\newblock {\em Commentarii Mathematici Helvetici}, 77:235--243, 2002.

\bibitem{Pinsky2023}
Tali Pinsky.
\newblock Analytical study of the lorenz system: Existence of infinitely many periodic orbits and their topological characterization.
\newblock {\em Proceedings of the National Academy of Sciences}, 120(31):e2205552120, 2023.

\bibitem{rolfsen2003knots}
Dale Rolfsen.
\newblock {\em Knots and links}, volume 346.
\newblock American Mathematical Soc., 2003.

\bibitem{sparrow2012lorenz}
Colin Sparrow.
\newblock {\em The Lorenz equations: bifurcations, chaos, and strange attractors}, volume~41.
\newblock Springer Science \& Business Media, 2012.

\bibitem{tucker1999lorenz}
Warwick Tucker.
\newblock The {L}orenz attractor exists.
\newblock {\em Comptes Rendus de l'Acad{\'e}mie des Sciences-Series I-Mathematics}, 328(12):1197--1202, 1999.

\end{thebibliography}
\newpage

\end{document}